\documentclass{article}

\usepackage{graphicx} 
\usepackage{amsmath, amssymb}
\usepackage{amscd}
\usepackage{verbatim}
\usepackage{tikz-cd}
\usepackage{graphicx} 
\usepackage[english]{babel}
\usepackage{amsthm}
\graphicspath{ {./images/} }

\usepackage[rightcaption]{sidecap}

\usepackage{wrapfig}
\usepackage{blindtext}
\usepackage[linktocpage]{hyperref}
\usepackage{titlesec}
\date{}
\newtheorem{theorem}{Theorem}[section]
\newtheorem{lemma}[theorem]{Lemma}
\newtheorem{remark}[theorem]{Remark}
\newtheorem{proposition}[theorem]{Proposition}
\newtheorem{definition}[theorem]{Definition}

\title{Differential calculus on Hopf--Galois extension via the Durdević braiding}

\author{Arnab Bhattacharjee \thanks{Mathematical Institute of Charles University, email: arnabbhatta7@gmail.comThe author acknowledges support from HORIZON-MSCA-2021-SE-01-CaLIGOLA and also from COST Action CaLISTA CA21109.}}

\begin{document}

\maketitle

\begin{abstract}
We introduce a class of right $H$--covariant first--order differential
calculi on principal comodule algebras generated by the Durdević
braiding $\sigma$ and a chosen vertical ideal.
Starting from the universal calculus, a strong connection, and a right
$H$--colinear splitting map, we construct $\sigma$--generated
differential calculi and prove their existence for arbitrary principal
comodule algebras. We show that, in this framework, universal vertical
maps and connection $1$--forms descend naturally to the quotient
calculus under suitable compatibility conditions. We further develop a
functorial formulation of $\sigma$--generated calculi and establish a
universal factorization property for the associated quotient calculi.
Finally, we present explicit examples arising from quantum projective
spaces and quantum lens spaces.
\end{abstract}

\tableofcontents

\section{Introduction}

Noncommutative differential geometry provides a natural framework for studying
geometric structures on quantum spaces arising from Hopf algebras and quantum
groups.
A central class of examples is furnished by quantum principal bundles, which
are algebraic counterparts of classical principal fiber bundles and are
modelled by Hopf--Galois extensions.
Since the foundational work of Brzeziński--Majid \cite{brzezinski1993quantum}, Hajac \cite{hajac1996strong}, Brzeziński--Hajac \cite{brzezinski2004chern} such extensions equipped
with strong connections have been recognized as appropriate noncommutative
analogues of principal bundles, supporting gauge theory, Chern--Connes
pairings, and index-theoretic constructions
\cite{dabrowski2001strong, montgomery2009hopf}.

In this framework, differential calculi play a fundamental role.
They encode infinitesimal geometric data and provide the setting for
connections, curvature, and gauge transformations.
Given an algebra $A$, every first--order differential calculus can be realized
as a quotient of the universal calculus $\Omega^1_{\mathrm{u}}(A)$ by a suitable
$A$--subbimodule \cite{woronowicz1989differential}.
For Hopf--Galois extensions, additional covariance and compatibility conditions
are required in order to reflect the underlying bundle structure.

An important contribution to this subject is due to Durdević, who introduced
a braided symmetry on quantum principal bundles in terms of a canonical
braiding on the balanced tensor product $A\otimes_{A^{\operatorname{co}H}}A$
\cite{durdevic1996quantum2}.
This braiding, now known as the Durdević braiding, governs the interaction
between horizontal and vertical directions and plays a central role in the
description of gauge transformations and braided symmetries
\cite{aschieri2023braided, aschieri2024braided}. 

Building on this idea, Durdević introduced the notion of a complete differential
calculus on a quantum principal bundle, consisting of a graded differential
algebra equipped with an extended coaction and satisfying strong compatibility
conditions \cite{durdevic1994geometry, durdevic1996geometry1}.
In this setting, the noncommutative Atiyah sequence is exact in all degrees and
admits a canonical horizontal--vertical decomposition.
More recently, this framework has been further analysed and refined in
\cite{del2025dhurdjevic}, where structural properties and applications were
investigated in detail.

While complete differential calculi provide a powerful and geometrically rich
theory, their construction requires substantial higher--order data and strong
regularity assumptions.
In many situations, however, one is primarily interested in first--order
structures and in the behaviour of connection forms and vertical maps.
This motivates the search for weaker frameworks that retain the essential
braided and descent properties while avoiding the full complexity of complete
calculi.

The purpose of this paper is to introduce and study a class of first--order
differential calculi on principal comodule algebras that are generated by the
Durdević braiding and a chosen vertical ideal.
More precisely, given a principal right $H$--comodule algebra $A$ with a strong
connection and a right ideal $I_H\subseteq H^+= \ker \varepsilon$, we consider differential calculi
whose defining relations are generated by the image of the associated universal
connection $1$--form and closed under the action of the Durdević braiding.
We refer to such calculi as \emph{$\sigma$--generated}, where $\sigma$ denotes here the Durdević braiding.

The notion of $\sigma$--generation isolates a minimal first--order structure
underlying the vertical geometry of Hopf--Galois extensions.
It ensures that connection $1$--forms and vertical maps descend from the
universal calculus to the quotient calculus and are compatible with the
intrinsic braided symmetry (parallel treatment can be found in \cite{del2025dhurdjevic}, Lemma 4.7).
At the same time, no horizontal--vertical decomposition or higher--order
compatibility is imposed. Moreover, the resulting construction admits a natural
functorial formulation with respect to morphisms of principal right
$H$--comodule algebras compatible with the underlying braided geometric
data.

The main results of this paper may be summarized as follows.
First, we show that on any principal comodule algebra, and for any fixed right
ideal $I_H\subseteq H^+$, there exists a $\sigma$--generated differential calculus.
This provides a systematic method for constructing braided-compatible
first--order calculi from strong connections.
We further show that universal vertical maps and
connection $1$--forms descend naturally to the associated quotient
calculi under suitable compatibility conditions. In addition, we develop
a functorial framework for $\sigma$--generated calculi and establish a
universal factorization property for the resulting quotient calculi.
Finally, we present explicit examples arising from quantum projective
spaces and quantum lens spaces, showing that $\sigma$--generated
calculi naturally appear in important classes of quantum principal
bundles.

Our approach differs from the complete calculus framework in that we work
exclusively at the level of first--order forms and formulate all conditions in
terms of the universal calculus, strong connections, and the induced braiding.
This allows us to treat arbitrary principal comodule algebras and to describe
explicitly the defining relations of the resulting calculi.
From this perspective, $\sigma$--generated calculi may be viewed as first--order
approximations of the complete calculus formalism, retaining its essential
braided and descent properties while remaining applicable in greater
generality.

\textbf{The paper is organized as follows.}
In Section~2, we recall basic notions concerning
differential calculi, Hopf--Galois extensions, principal comodule
algebras, strong connections, and the Durdević braiding.
In Section~3, we introduce $\sigma$--generated differential calculi and
establish their structural, descent, and functorial properties,
including the universal factorization property of the associated
quotient calculi.
Finally, in Section~4, we present explicit examples arising from quantum
projective spaces and quantum lens spaces.


\section{Preliminaries}
We fix a field $k$. Throughout this paper, all spaces are $k$-vector spaces and all maps are $k$-linear unless otherwise stated. The tensor product of $k$-spaces is denoted by $\otimes$, and the balanced tensor product over an algebra $B$ is denoted by $\otimes_{B}$. We denote $H$ to be a Hopf algebra with coproduct $\Delta$, counit $\varepsilon$ and a bijective antipode $S$. An algebra $A$ is always assumed to be associative and unital. We use Sweedler notation for coproducts and coactions, suppressing summation symbols. In this section we follow the literature from \cite{beggs2020quantum, lecturenotesncg, durdevic1996quantum2}. 

\subsection{Universal and first--order differential calculi}

The \emph{universal first--order differential calculus} over an algebra $A$ is defined as
\[
\Omega^1_{\mathrm{u}}(A) := \ker(m:A\otimes A\to A),
\]
with universal differential $\mathrm{d}_{\mathrm{u}}:A\to\Omega^1_{\mathrm{u}}(A)$
given by
\[
\mathrm{d}_{\mathrm{u}}(a)=1\otimes a - a\otimes 1.
\]
A \emph{first--order differential calculus} over $A$ is a pair
$(\Omega^1(A),\mathrm{d})$ such that $\Omega^1(A)$ is an $A$--bimodule and
$\mathrm{d}:A\to\Omega^1(A)$ satisfies the Leibniz rule and
$\Omega^1(A)=\mathrm{span}\{a\,\mathrm{d}b : a,b\in A\}$.
Every first--order differential calculus arises as a quotient
\[
\Omega^1(A)=\Omega^1_{\mathrm{u}}(A)/N_A,
\]
for some $A$--subbimodule $N_A\subset\Omega^1_{\mathrm{u}}(A)$.

\subsection{Principal comodule algebra}

Let $H$ be a Hopf algebra and $A$ be a right $H$--comodule algebra with coaction
$\delta:A\to A\otimes H$ and is given by $\delta(a)= a_{(0)}\otimes a_{(1)}$.
The subalgebra of coinvariants is
\[
B:=A^{\operatorname{co}(H)}=\{\,b\in A \mid \delta(b)=b\otimes 1\,\}.
\]
An extension $B\subset A$ is called a \emph{Hopf--Galois extension} if the
canonical map
\[
\mathrm{can}:A\otimes_B A \longrightarrow A\otimes H,
\qquad
a\otimes_B a'\longmapsto a\,\delta(a')= aa'_{(0)}\otimes a'_{(1)},
\]
is bijective.
In this case one defines the \emph{translation map}
\[
\tau:H\longrightarrow A\otimes_B A,
\qquad
\tau(h)=\mathrm{can}^{-1}(1\otimes h),
\qquad
\tau(h)=h^{\langle 1\rangle}\otimes_B h^{\langle 2\rangle}.
\]

\begin{definition}
    A right $H$--comodule algebra $A$ is called a \emph{principal comodule algebra} 
if:
\begin{enumerate}
\item $B=A^{\operatorname{co}(H)}\subset A$ is a Hopf--Galois extension;
\item the multiplication map $\mu: B\otimes A\to A$ admits a left $B$--linear, right
$H$--colinear splitting $s: A\rightarrow B\otimes A$ such that $\mu \circ s= \mathrm{id}_{A}$.
\end{enumerate}
\end{definition}

\subsection{Vertical map and strong connection}

Let $B= A^{\mathrm{co}(H)}\subset A$ be a Hopf--Galois extension. 
\begin{definition}
The \emph{vertical map} on a universal calculus is a linear map
\[
\mathrm{ver}_{\mathrm{u}}:\Omega^1_{\mathrm{u}}(A)\longrightarrow A\otimes H^+
\]
defined by
\[
\mathrm{ver}_{\mathrm{u}}(a\otimes a')= a a'_{(0)}\otimes (a'_{(1)})^+,
\]
where $H^+=\ker\varepsilon$.

\end{definition}

\begin{definition}
A \emph{strong connection} on a Hopf-Galois extension $B= A^{\mathrm{co}(H)}\subset A$ is a linear map
\[
\ell:H\longrightarrow A\otimes A
\]
satisfying 
\begin{enumerate}
    \item $\ell(1_{H})= 1_{A}\otimes 1_{A}$;
    \item $(\mathrm{id}_{A}\otimes \delta)\circ \ell= (\ell\otimes \mathrm{id}_{H})\circ \Delta$;
    \item $\widetilde{\mathrm{can}}\circ \ell= 1_{A}\otimes \mathrm{id}_{H}$, where $\widetilde{\mathrm{can}}: A\otimes A\rightarrow A\otimes H$ be the lift of $\mathrm{can}$ to $A\otimes A$.
\end{enumerate}
\end{definition}

\begin{definition}\label{def:connection-1-form}
Let $\ell:H\to A\otimes A$ be a strong connection.
The associated \emph{connection $1$--form} is the linear map
\[
\omega:H^+\longrightarrow\Omega^1_{\mathrm{u}}(A),
\qquad
\omega(h^+)=\ell(h)-\varepsilon(h)\,1_A\otimes 1_A,
\]
where $H^+=\ker\varepsilon$, and $h^+:= h- \varepsilon(h)1_{H}$.
It satisfies the following properties:
\begin{enumerate}
\item $\mathrm{ver}_{\mathrm{u}}\circ \omega
= 1_A\otimes \mathrm{id}_{H^+}$;

\item $\delta_{\Omega^1_{\mathrm{u}}(A)}\circ \omega
= (\omega\otimes \mathrm{id}_{H})\circ \mathrm{Ad}_{R}$,
where $\delta_{\Omega^1_{\mathrm{u}}(A)}$ denotes the right $H$--coaction on
$\Omega^1_{\mathrm{u}}(A)$ and $\mathrm{Ad}_{R}$ denotes the right adjoint
coaction of $H$ on $H^+$.
\end{enumerate}
\end{definition}

By [\cite{brzezinski2004chern}, Theorem 2.5], every principal comodule algebra admits a strong
connection.

\subsection{Durdević braiding}\label{durdevic}

Let $B= A^{\mathrm{co}(H)}\subset A$ be a Hopf--Galois extension.
Using the translation map $\tau$, one defines the \emph{Durdević braiding}
\[
\sigma:A\otimes_B A \longrightarrow A\otimes_B A
\]
by
\[
\sigma(a\otimes_B a')
=
a_{(0)}a'\, \tau(a_{(1)})= a_{(0)}a'(a_{(1)})^{\langle 1 \rangle}\otimes_{B} (a_{(1)})^{\langle 2 \rangle}.
\]
\begin{proposition}[\cite{durdevic1996quantum2}, Proposition 2.1]
    Let $B= A^{\mathrm{co}(H)}\subset A$ be a Hopf-Galois extension. Then the following hold true:
    \begin{enumerate}
        \item The map $\sigma: A\otimes_{B}A\rightarrow A\otimes_{B}A$ is an isomorphism. Its inverse is
        $$\sigma^{-1}(a\otimes_{B}a')= \tau(S^{-1}(a'_{(1)}))aa'_{(0)}= S^{-1}(a'_{(1)})^{\langle 1 \rangle} \otimes_{B} S^{-1}(a'_{(1)})^{\langle 2\rangle}aa'_{(0)}$$
        \item $\sigma$ satisfies the braid relation:
        \[
(\sigma\otimes_{B}\mathrm{id})\circ(\mathrm{id}\otimes_{B}\sigma)\circ(\sigma\otimes_{B}\mathrm{id})
=
(\mathrm{id}\otimes_{B}\sigma)\circ(\sigma\otimes_{B}\mathrm{id})\circ(\mathrm{id}\otimes_{B}\sigma)
\]
on $A\otimes_B A\otimes_B A$;
\item $A$ is $\sigma$-commutative, in other words,
$$m_{A}\circ \sigma= m_{A}$$ where $m_{A}$ denotes the multiplication $A\otimes_{B}A\rightarrow A$;
\item The product map $m_A$ is compatible with braiding $\sigma$, that is

\begin{enumerate}
    \item $\sigma\circ (m_A\otimes_{B} \mathrm{id})= (\mathrm{id}\otimes_{B}m_{A})\circ (\sigma\otimes_{B}\mathrm{id})\circ (\mathrm{id}\otimes_{B}\sigma)$
    \item $\sigma\circ (\mathrm{id}\otimes_{B} m_A)= (m_A\otimes_{B}\mathrm{id})\circ (\mathrm{id}\otimes_{B}\sigma)\circ (\sigma\otimes_{B}\mathrm{id})$
\end{enumerate}
    \end{enumerate}
\end{proposition}

\section{Differential calculus generated by the Durdević braiding}

The Durdević braiding encodes an intrinsic braided symmetry associated
to a Hopf--Galois extension and plays a fundamental role in the geometry
of quantum principal bundles. In this section we study how this braided
structure interacts with first--order differential calculi on principal
comodule algebras.

Our goal is to construct a class of right $H$--covariant differential
calculi generated by braided vertical relations arising from strong
connections and the Durdević braiding. We begin with establishing several
structural properties of the braiding and of the associated vertical
subspaces in $A\otimes_{B}A$.

\medskip

Since the Durdević braiding acts naturally on the balanced tensor
product $A\otimes_BA$, it is natural to consider subspaces that are stable under this braiding. Such subspaces will play a central role in the construction
of the associated differential calculi.

\begin{definition}
\label{def:sigma-generated}
Let $A$ be a right $H$--comodule algebra and $B=A^{\mathrm{co}(H)}\subset A$ be a Hopf--Galois extension with
structure Hopf algebra $H$. We consider $$\sigma:A\otimes_B A\longrightarrow A\otimes_B A$$ be the Durdević braiding. A subspace $V\subseteq A\otimes_{B}A$ is said to be \emph{$\sigma$--stable} if $\sigma(V)= V$.

For a subspace $W\subseteq A\otimes_B A$, denote by $\langle W\rangle_\sigma$ the smallest $\sigma$--stable subspace of $A\otimes_B A$
containing $W$. 
\end{definition}

\begin{proposition}
\label{prop:braided-sigma-generated}
Let $A$ be a right $H$--comodule algebra and $B= A^{\mathrm{co}(H)}\subset A$ be a Hopf-Galois extension with structure Hopf algebra $H$. Then for any subspace $W\subseteq A\otimes_{B}A$, the Durdević braiding $\sigma$ induces a well-defined isomorphism
\[
\bar{\sigma}:(A\otimes_B A)/\langle W\rangle_\sigma
\longrightarrow (A\otimes_B A)/\langle W\rangle_\sigma
\]
\end{proposition}

\begin{proof}

By construction, for any subspace $W\subseteq A\otimes_{B}A$, $\langle W\rangle_\sigma$ is an $\sigma$-stable subspace of $A\otimes_{B}A$, it follows that
the braiding $\sigma$ descends to a well-defined endomorphism
\[
\bar{\sigma}:(A\otimes_B A)/\langle W\rangle_\sigma
\longrightarrow (A\otimes_B A)/\langle W\rangle_\sigma.
\]
is defined by 
\[
[a\otimes_{B}a']_{\langle W\rangle_\sigma}\longmapsto [\sigma(a\otimes_{B}a')]_{\langle W\rangle_\sigma}
\]

The injectivity of $\bar{\sigma}$ follows from the fact that $\langle W\rangle_\sigma$ is an $\sigma$--stable subspace and the braiding $\sigma$ is an isomorphism. Moreover, surjectivity of $\bar{\sigma}$ follows directly from the surjectivity of the braiding $\sigma$. Therefore, $\bar{\sigma}$ is an isomorphism.
\end{proof}

Proposition~\ref{prop:braided-sigma-generated} shows that the Durdević braiding descends naturally to
quotients by $\sigma$--stable subspaces. This provides the basic
mechanism allowing braided relations to be imposed consistently at the
level of quotient spaces.

\medskip

In order to construct right $H$--covariant quotient calculi compatible
with the braided structure, it is important to understand the behaviour
of the Durdević braiding with respect to the underlying $H$--coaction.
The following lemma shows that the braiding is compatible with the
natural diagonal coaction on $A\otimes_BA$.
\begin{lemma}\label{lem:colinearitydurdevic}
Let $A$ be a right $H$--comodule algebra and $B=A^{\operatorname{co}(H)}\subset A$ be a Hopf--Galois extension with
translation map
\[
\tau:H\longrightarrow A\otimes_B A,
\qquad
\tau(h)=h^{\langle 1\rangle}\otimes_B h^{\langle 2\rangle}.
\]
Let $\sigma$ be the Durdević braiding,
then $\sigma$ is right $H$--colinear with respect to the diagonal right
$H$--coaction
\[
\delta_{A\otimes_B A}(a\otimes_B a')
=
a_{(0)}\otimes_B a'_{(0)}
\otimes
a_{(1)}a'_{(1)}.
\]
That is,
\[
\delta_{A\otimes_B A}\circ \sigma
=
(\sigma\otimes \mathrm{id}_H)\circ \delta_{A\otimes_B A}.
\]
\end{lemma}

\begin{proof}
Let $\sum_{i}a_{i}\otimes_{B} a'_{i}\in A\otimes_{B}A$, where summation is over finitely many elements $a_i, a'_i\in A$.

We first compute the left-hand side. Using the definition of $\sigma$,
\[
\sigma(a\otimes_B a')
=
a_{(0)}a'(a_{(1)})^{\langle 1\rangle}
\otimes_B
(a_{(1)})^{\langle 2\rangle}.
\]
Applying the diagonal coaction, we obtain
\begin{align*}
&
\delta_{A\otimes_B A}\bigl(\sigma(\sum_{i}a_{i}\otimes_B a'_{i})\bigr)
\\
&=
\delta_{A\otimes_B A}
\Bigl(\sum_{i}
a_{i(0)}a'_{i}(a_{i(1)})^{\langle 1\rangle}
\otimes_B
(a_{i(1)})^{\langle 2\rangle}
\Bigr)
\\
&=
\sum_{i}\Bigl(
a_{i(0)}a'_{i}(a_{i(1)})^{\langle 1\rangle}
\Bigr)_{(0)}
\otimes_B
\Bigl(
(a_{i(1)})^{\langle 2\rangle}
\Bigr)_{(0)}
\\
&\qquad\qquad\otimes
\Bigl(
a_{i(0)}a'_{i}(a_{i(1)})^{\langle 1\rangle}
\Bigr)_{(1)}
\Bigl(
(a_{i(1)})^{\langle 2\rangle}
\Bigr)_{(1)}.
\end{align*}
Since $\delta:A\to A\otimes H$ is an algebra map, this becomes
\begin{align*}
&=\sum_{i}
a_{i(0)}a'_{i(0)}
\bigl((a_{i(1)})^{\langle 1\rangle}\bigr)_{(0)}
\otimes_B
\bigl((a_{i(1)})^{\langle 2\rangle}\bigr)_{(0)}
\\
&\qquad\qquad\otimes
a_{i(1)}a'_{i(1)}
\bigl((a_{i(1)})^{\langle 1\rangle}\bigr)_{(1)}
\bigl((a_{i(1)})^{\langle 2\rangle}\bigr)_{(1)}.
\end{align*}

Now we use the standard identity satisfied by the translation map \cite[Proposition 3.6]{brzezinski1996translation}: For $h\in H$,
\begin{align}
h^{\langle 1\rangle}
\otimes_{B}
(h^{\langle 2\rangle})_{(0)}
\otimes
(h^{\langle 2\rangle})_{(1)}
&=
(h_{(1)})^{\langle 1\rangle}
\otimes_{B}
(h_{(1)})^{\langle 2\rangle}
\otimes
h_{(2)},
\\[6pt]
(h^{\langle 1\rangle})_{(0)}
\otimes_{B}
h^{\langle 2\rangle}
\otimes
(h^{\langle 1\rangle})_{(1)}
&=
(h_{(2)})^{\langle 1\rangle}
\otimes_{B}
(h_{(2)})^{\langle 2\rangle}
\otimes
S(h_{(1)}).
\end{align}
Applying this identity with $h=a_{i(1)}$, we obtain
\begin{align}
    \delta_{A\otimes_{B}A}(\sigma(\sum_{i}a_{i}\otimes_{B}a'_{i}))
    &=
   \sum_{i} a_{i(0)}a'_{i(0)}((a_{i(1)})^{\langle 1 \rangle})_{(0)}\otimes_{B} ((a_{i(1)})^{\langle 2 \rangle})_{(0)}\otimes a_{i(1)}\varepsilon(a_{i(1)})a'_{i(1)}
\end{align}
Using the compatibility condition of the coproduct of $H$ and the coaction $\delta$,
\[
a_{i(0)(0)}\otimes a_{i(0)(1)}\otimes a_{i(1)}
=
a_{i(0)}\otimes a_{i(1)}\otimes a_{i(2)},
\]
This implies, $a_{i(1)}= a_{i(2)}$ by compairing third leg of the tensor product, therefore, from $(3)$ and using the counit identity, we get, 
\begin{align*}
&
\delta_{A\otimes_B A}\bigl(\sigma(\sum_{i}a_{i}\otimes_B a'_{i})\bigr)
\\
&=
\sum_{i} a_{i(0)}a'_{i(0)}(a_{i(1)})^{\langle 1 \rangle}\otimes_{B} (a_{i(1)})^{\langle 2 \rangle}\otimes a_{i(1)}a'_{i(1)}
\end{align*}
Next we compute the right-hand side:
\begin{align*}
&
(\sigma\otimes \mathrm{id}_H)
\delta_{A\otimes_B A}(\sum_{i}a_{i}\otimes_B a'_{i})
\\
&=
(\sigma\otimes \mathrm{id}_H)
\Bigl(\sum_{i}
a_{i(0)}\otimes_B a'_{i(0)}
\otimes
a_{i(1)}a'_{i(1)}
\Bigr)
\\
&=\sum_{i}
\sigma(a_{i(0)}\otimes_B a'_{i(0)})
\otimes
a_{i(1)}a'_{i(1)}
\\
&=\sum_{i}
a_{i(0)(0)}a'_{i(0)}
(a_{i(0)(1)})^{\langle 1\rangle}
\otimes_B
(a_{i(0)(1)})^{\langle 2\rangle}
\otimes
a_{i(1)}a'_{i(1)}.
\end{align*}
Again using the compatibility condition of the coproduct of $H$ and the coaction $\delta$,
\[
a_{i(0)(0)}\otimes a_{i(0)(1)}\otimes a_{i(1)}
=
a_{i(0)}\otimes a_{i(1)}\otimes a_{i(2)},
\]
we obtain
\[
(\sigma\otimes \mathrm{id}_H)
\delta_{A\otimes_B A}(\sum_{i}a_{i}\otimes_B a'_{i})
=\sum_{i}
a_{i(0)}a'_{i(0)}
(a_{i(1)})^{\langle 1\rangle}
\otimes_B
(a_{i(1)})^{\langle 2\rangle}
\otimes
a_{i(1)}a'_{i(1)}.
\]
Hence
\[
\delta_{A\otimes_B A}\circ \sigma
=
(\sigma\otimes \mathrm{id}_H)\circ \delta_{A\otimes_B A}.
\]
Therefore $\sigma$ is right $H$--colinear.
\end{proof}

We now establish several technical results relating the vertical map,
the canonical Hopf--Galois map, and the connection $1$--form associated
to a strong connection. These results will be used in the construction
of $\sigma$--generated differential calculi.

Let $B=A^{\mathrm{co}(H)}\subset A$ is a Hopf--Galois extension. 
Denote by
\[
\pi:A\otimes A\longrightarrow A\otimes_B A
\]
the canonical quotient map.

\medskip

\begin{lemma}\label{lem:ver-can-identification}
Let $A$ be a right $H$--comodule algebra and $B= A^{\mathrm{co}(H)}\subset A$ be a Hopf--Galois extension with structure Hopf algebra $H$. Let $I_H\subseteq H^+$ be a right ideal.
Let
\[
\mathrm{ver}_{\mathrm{u}}:\Omega^1_{\mathrm{u}}(A)\longrightarrow A\otimes H^+
\]
be the vertical map, and
\[
\mathrm{can}:A\otimes_B A\longrightarrow A\otimes H
\]
the Hopf--Galois canonical map.
Then
\[
\pi\big(\mathrm{ver}_{\mathrm{u}}^{-1}(A\otimes I_H)\big)
\;=\;
\mathrm{can}^{-1}(A\otimes I_H)
\;\subset\;
A\otimes_B A .
\]
where $\mathrm{ver}_{\mathrm{u}}^{-1}(A\otimes I_{H})=\lbrace x\in \Omega^1_{\mathrm{u}}(A): \mathrm{ver}_{\mathrm{u}}(x)\in A\otimes I_{H}\rbrace$.
\end{lemma}

\begin{proof}
Let $\mathrm{pr}_{H^+}:H\to H^+$ denote the canonical projection
$\mathrm{pr}_{H^+}(h)=h-\varepsilon(h)1_{H}$.
By definition of the vertical map one has
\[
\mathrm{ver}_{\mathrm{u}}
=
(\mathrm{id}\otimes \mathrm{pr}_{H^+})\circ \mathrm{can}\circ \pi 
\tag{4}
\]

Let $x\in\Omega^1_{\mathrm{u}}(A)$.
Since $\Omega^1_{\mathrm{u}}(A)=\ker(m)$, it follows that
\[
(\mathrm{id}\otimes\varepsilon)\big(\mathrm{can}(\pi(x))\big)=0
\tag{5}
\]
Indeed, for $x=a'\,\mathrm{d}_{\mathrm{u}}a$ one has
\(
\mathrm{can}(\pi(x))
=
a'a_{(0)}\otimes a_{(1)}-a'a\otimes 1
\),
and applying $\mathrm{id}\otimes\varepsilon$ yields zero.

Combining $(4)$ and $(5)$ gives
\[
\mathrm{ver}_{\mathrm{u}}(x)=\mathrm{can}(\pi(x)).
\tag{6}
\]

Since $I_H\subseteq H^+$, equality $(6)$ implies
\[
\mathrm{ver}_{\mathrm{u}}(x)\in A\otimes I_H
\quad\Longleftrightarrow\quad
\mathrm{can}(\pi(x))\in A\otimes I_H .
\]

Applying $\pi$ and using the bijectivity of $\mathrm{can}$ yields
\[
\pi\big(\mathrm{ver}_{\mathrm{u}}^{-1}(A\otimes I_H)\big)
=
\mathrm{can}^{-1}(A\otimes I_H),
\]
as claimed.
\end{proof}

\begin{lemma}
\label{rem:connection-in-vertical}
Let $A$ be a principal right $H$--comodule algebra and $B= A^{\mathrm{co}(H)} \subset A$ be a Hopf--Galois extension with structure Hopf algebra $H$. Let $I_H \subseteq H^+$ be a right ideal of $H$.
Let $\omega:H^+\to\Omega^1_{\mathrm{u}}(A)$ be the connection
$1$-form associated to a strong connection on $A$.
Then
\[
\pi(\omega(I_H)) \subset 
\mathrm{can}^{-1}(A\otimes I_H).
\]
\end{lemma}

\begin{proof}
For any $h^+\in I_H$, the defining property of a connection 1-form gives
\[
\mathrm{ver}_{\mathrm{u}}(\omega(h^+)) = 1\otimes h^+ \in A\otimes I_H.
\]
Hence $\omega(h^+)\in \mathrm{ver}_{\mathrm{u}}^{-1}(A\otimes I_H)$, and applying the
canonical quotient map $\pi$ yields
\[
\pi(\omega(h^+))\in \pi\big(\mathrm{ver}_{\mathrm{u}}^{-1}(A\otimes I_H)\big).
\]
Therefore, by the Lemma~\ref{lem:ver-can-identification} yields
\[
\pi(\omega(h^+))\in \mathrm{can}^{-1}(A\otimes I_H).
\]
Since $h^+\in I_H$ was arbitrary, the claim follows.
\end{proof}

Lemma~\ref{rem:connection-in-vertical} shows that the image of the connection $1$--form determines a
distinguished vertical subspace inside $A\otimes_{B}A$ naturally associated to the chosen right ideal $I_{H}\subseteq H^+$.
The next step is to study the behaviour of the corresponding
$\sigma$--stable closure under the right $H$--coaction.

\medskip

It is clear from the Lemma~\ref{rem:connection-in-vertical} that $\pi(\omega(I_{H}))\subset \mathrm{can}^{-1}(A\otimes I_H)$. Therefore, by Definition~\ref{def:sigma-generated}, $\langle \pi(\omega(I_{H}))\rangle_{\sigma}$ is the smallest $\sigma$--stable subspace of $A\otimes_{B}A$ containing the subspace $\pi(\omega(I_{H}))$.

\begin{lemma}\label{lem:Nbalcovariant}
Let $A$ be a principal right $H$--comodule algebra. We consider $B= A^{\mathrm{co}(H)}\subset A$ be a Hopf-Galois extension and let
$I_H\subseteq H^+$ be a right ideal invariant under the right adjoint
coaction of $H$.
Let $\omega:H^+\to\Omega^1_{\mathrm{u}}(A)$ be the connection $1$--form
associated to a strong connection on $A$.
Then $\langle \pi(\omega(I_H))\rangle_\sigma
\subset A\otimes_{B}A$ is invariant under the right $H$--coaction, that is,
\[
\delta_{A\otimes_{B}A}(\langle \pi(\omega(I_H))\rangle_\sigma
)
\subset
\langle \pi(\omega(I_H))\rangle_\sigma
\otimes H.
\]
\end{lemma}

\begin{proof}
By Definition~\ref{def:connection-1-form}, the connection $1$--form $\omega$ satisfies
\[
\delta_{\Omega^1_{\mathrm{u}}(A)}\circ\omega
=
(\omega\otimes\mathrm{id})\circ\mathrm{Ad}_R,
\]
where $\delta_{\Omega^1_{\mathrm{u}}(A)}$ denotes the right $H$--coaction on $\Omega^1_{\mathrm{u}}(A)$ and
$\mathrm{Ad}_R$ is the right adjoint coaction of $H$.

Since $I_H$ is invariant under $\mathrm{Ad}_R$, we have
\[
\mathrm{Ad}_R(I_H)\subset I_H\otimes H.
\]
Hence,
\[
\delta_{\Omega^1_{\mathrm{u}}(A)}(\omega(I_H))
=
(\omega\otimes\mathrm{id})(\mathrm{Ad}_R(I_H))
\subset
\omega(I_H)\otimes H.
\]
Therefore, $\omega(I_H)$ is a right $H$--subcomodule of $\Omega^1_{\mathrm{u}}(A)$.

Let
\[
\pi:\Omega^1_{\mathrm{u}}(A)\longrightarrow A\otimes_{B}A
\]
be the canonical projection.
Since $\pi$ is right $H$--colinear, it follows that
\[
\delta_{A\otimes_{B} A}(\pi(\omega(I_H)))
=
(\pi\otimes\mathrm{id})(\delta_{\Omega^1_{\mathrm{u}}(A)}(\omega(I_H)))
\subset
\pi(\omega(I_H))\otimes H.
\]
Thus, $\pi(\omega(I_H))$ is a right $H$--subcomodule of
$A\otimes_{B}A$.

Moreover, by Lemma~\ref{lem:colinearitydurdevic}, the Durdević braiding $\sigma$ is right $H$--colinear.
Consequently, for all $x\in A\otimes_{B}A$,
\[
\delta_{A\otimes_{B}A}(\sigma(x))
=
(\sigma\otimes\mathrm{id}_{H})(\delta_{A\otimes_{B}A}(x)).
\]
It follows that the right $H$--coaction preserves the $\sigma$--stable
subspace generated by $\pi(\omega(I_H))$.

Since $\langle \pi(\omega(I_H))\rangle_\sigma$ is, by definition, the smallest $\sigma$--stable
subspace containing $\pi(\omega(I_H))$, we conclude that
\[
\delta_{A\otimes_{B}A}(\langle \pi(\omega(I_H))\rangle_\sigma
)
\subset
\langle \pi(\omega(I_H))\rangle_\sigma
\otimes H.
\]
This completes the proof.
\end{proof}

\subsection{$\sigma$--generated differential calculi}

We now introduce the main class of differential calculi considered in
this paper. Roughly speaking, these calculi are generated by the
braided vertical relations arising from the image of the connection
$1$--form together with the Durdević braiding. The resulting quotient
calculi provide first--order differential structures naturally adapted
to the braided geometry of principal comodule algebras.

\medskip

Let $A$ be a principal right $H$--comodule algebra and $B= A^{\mathrm{co}(H)}\subset A$ be a Hopf-Galois extension. For every subspace $N\subseteq A\otimes A$, we define the \emph{balanced image} of $N\subseteq A\otimes A$ under the cannonical quotient map $\pi$ by
\[
N^{\mathrm{bal}}
\;:=\;
\pi(N)
\subset\;
A\otimes_B A.
\]

The following definition formalizes the idea that the defining vertical
relations of the calculus should be generated by the $\sigma$--stable
closure of the image of the connection $1$--form associated to the
chosen ideal $I_H\subseteq H^+$.
\begin{definition}\label{def:sigma-generatedcalculi}
Let $A$ be a principal right $H$-comodule algebra, let
$I_H\subseteq H^+$ be a fixed right ideal of $H$, and let
\[
\Omega^1(A)=\Omega^1_{\mathrm{u}}(A)/N_A
\]
be a right $H$-covariant first-order differential calculus on $A$, where $N_{A}\subset \Omega^1_{\mathrm{u}}(A)$ be a right $H$-covariant $A$-subbimodule satisfying
\[
\mathrm{ker}\,\pi\subset N_A
\]
The calculus $\Omega^1(A)$ is called \emph{$\sigma$-generated with respect to
$I_H$} if there exists a strong connection 
$\ell:H\to A\otimes A$ with associated connection $1$-form $\omega$
such that
\[
N_A^{\mathrm{bal}}
:=
A\cdot \langle \pi(\omega(I_H)) \rangle_\sigma \cdot A ,
\qquad \text{where  }
N_A^{\mathrm{bal}}:=\pi(N_A).
\]
\end{definition}

Having introduced the notion of a $\sigma$--generated calculus, the
first natural question is whether such calculi exist for arbitrary
principal comodule algebras. The following proposition provides a
general existence result under the presence of an equivariant splitting
of the canonical short exact sequence.
\begin{proposition}\label{prop:existence-sigma-generated}
Let $A$ be a principal right $H$--comodule algebra. We consider $B= A^{\mathrm{co}(H)}\subset A$ be a Hopf-Galois extension and let
$I_H\subseteq H^+$ be a right ideal invariant under the right adjoint
coaction of $H$.
Assume that there exists a right $H$--colinear map
\[
s:A\otimes_{B}A\longrightarrow \Omega^1_{\mathrm{u}}(A)
\]
such that $\pi\circ s=\mathrm{id}$, where
$\pi:\Omega^1_{\mathrm{u}}(A)\to A\otimes_{B}A$ is the canonical
quotient map.
Then there exists a right $H$--covariant first--order differential calculus
\[
\Omega^1(A)=\Omega^1_{\mathrm{u}}(A)/N_A
\]
on $A$ which is $\sigma$--generated with respect to $I_H$.
\end{proposition}

\begin{proof}

Consider the short exact sequence of right $H$--comodules
\[
0\longrightarrow A\Omega^1_{\mathrm{u}}(B)A
\longrightarrow \Omega^1_{\mathrm{u}}(A)
\xrightarrow{\;\pi\;}
A\otimes_{B}A
\longrightarrow 0.
\]
By assumption, this sequence admits a right $H$--colinear splitting $s$.

Define
\[
N_A:= A \cdot s(\langle \pi(\omega(I_H))\rangle_\sigma)\cdot A +  A\Omega^1_{\mathrm{u}}(B)A
\subset \Omega^1_{\mathrm{u}}(A).
\]
Since by assumption $s$ is right $H$--colinear, and by Lemma~\ref{lem:Nbalcovariant}, $\langle \pi(\omega(I_H))\rangle_\sigma$ is a right
$H$--subcomodule, it follows that $s(\langle \pi(\omega(I_H))\rangle_\sigma)$ is right $H$--subcomodule, and hence,
$A \cdot s(\langle \pi(\omega(I_H))\rangle_\sigma)\cdot A$ is a right $H$--subcomodule. Explicitly, it is given by 

\[
\delta_{A\cdot s(\langle \pi(\omega(I_H))\rangle_\sigma)\cdot A}:
A\cdot s(\langle \pi(\omega(I_H))\rangle_\sigma)\cdot A
\longrightarrow
A\cdot s(\langle \pi(\omega(I_H))\rangle_\sigma)\cdot A \otimes H,
\]
defined by
\[
a m a'
\longmapsto
a_{(0)}m_{(0)}a'_{(0)}
\otimes
a_{(1)}m_{(1)}a'_{(1)},
\]
for all $a,a'\in A$ and $m\in s(\langle \pi(\omega(I_H))\rangle_\sigma)$.

Moreover, $A\Omega^1_{\mathrm{u}}(B)A$ is a right
$H$--subcomodule, therefore, $N_A$ is a right $H$--subcomodule of
$\Omega^1_{\mathrm{u}}(A)$.

 Since both $ A \cdot s(N_A^{\mathrm{bal}})\cdot A$ and $A\Omega^1_{\mathrm{u}}(B)A$ are $A$--subbimodules of $\Omega^1_{\mathrm{u}}(A)$. It follows that, $N_A$ is an $A$--subbimodule of $\Omega^1_{\mathrm{u}}(A)$. Moreover, $\mathrm{ker}\, \pi= A\Omega^1_{\mathrm{u}}(B)A\subset N_A$.
 
 Therefore, we define the quotient
\[
\Omega^1(A):=\Omega^1_{\mathrm{u}}(A)/N_A.
\]

Because $N_A$ is a right $H$--covariant $A$--subbimodule, the quotient
$\Omega^1(A)$ inherits a right $H$--coaction and hence is a right
$H$--covariant first--order differential calculus on $A$.

Finally, we know that $\pi$ is an $A$--bimodule map, and $\pi\circ s=\mathrm{id}$, therefore, by construction,
\[
\pi(N_A)=A\cdot \langle \pi(\omega(I_H))\rangle_\sigma \cdot A
= N_A^{\mathrm{bal}},
\]
which shows that $\Omega^1(A)$ is $\sigma$--generated with respect to
$I_H$.
This completes the proof.
\end{proof}

\begin{remark}
The existence of a right $H$--colinear splitting map
\[
s:A\otimes_B A\longrightarrow \Omega^1_{\mathrm{u}}(A)
\]
means that the short exact sequence of right $H$--comodules
\[
0\longrightarrow A\Omega^1_{\mathrm{u}}(B)A
\longrightarrow \Omega^1_{\mathrm{u}}(A)
\xrightarrow{\;\pi\;}
A\otimes_B A
\longrightarrow 0
\]
splits equivariantly.

Such equivariant splittings are closely related to the existence of
horizontal lifts and strong connections on Hopf--Galois extensions in the
sense of Brzezi\'nski--Majid and Hajac
\cite{brzezinski1998coalgebra,hajac1996strong}.

Examples where such equivariant structures naturally occur include
trivial and cleft Hopf--Galois extensions, crossed products, and strongly
graded principal comodule algebras
\cite{bohm2006cleft,lecturenotesncg,brzezinski1998coalgebra}.
\end{remark}

One of the fundamental geometric structures associated to a
Hopf--Galois extension is the vertical map $\mathrm{ver}_{\mathrm{u}}$.
We now show that, under a natural compatibility condition, this map
descends from the universal calculus to the associated
$\sigma$--generated quotient calculus.

\begin{proposition}\label{prop:verdescent}
Let $A$ be a principal right $H$--comodule algebra and $B= A^{\mathrm{co}(H)}\subset A$ be a Hopf-Galois extension. Let
$I_H\subseteq H^+$ be a right ideal invariant under the right adjoint coaction of $H$.
Assume that there exists a right $H$--colinear splitting map
\[
s:A\otimes_BA\to\Omega^1_{\mathrm{u}}(A)
\]
such that
\[
\mathrm{ver}_{\mathrm{u}}
\Bigl(
s(\langle\pi(\omega(I_H))\rangle_\sigma)
\Bigr)
\subseteq
A\otimes I_H.
\]

Then the vertical map $\mathrm{ver}_{\mathrm{u}}$ descends to a well--defined surjective
map
\[
\mathrm{ver}:
\Omega^1(A)\to A\otimes H^+/I_H,
\]
where
\[
\Omega^1(A)=\Omega^1_{\mathrm{u}}(A)/N_A
\]
is the corresponding $\sigma$--generated differential calculus.
\end{proposition}

\begin{proof}
By the construction in the proof of the Proposition~\ref{prop:existence-sigma-generated}, we define
\[
N_A
:=
A\cdot s(\langle\pi(\omega(I_H))\rangle_\sigma)\cdot A
+
A\Omega^1_{\mathrm{u}}(B)A.
\]
It follows that $\Omega^1(A)= \Omega^1_{\mathrm{u}}(A)/N_A$ is an $\sigma$--generated differential calculus.

\medskip

In order to show that $\mathrm{ver_{\mathrm{u}}}$ descend to a well-defined map on $\Omega^1(A)$, we shall show that $\mathrm{ver}_{\mathrm{u}}(N_A)\subseteq A\otimes I_{H}$.
Since
\[
\mathrm{ver}_{\mathrm{u}}(A\Omega^1_{\mathrm{u}}(B)A)=0,
\]
it suffices to prove that
\[
\mathrm{ver}_{\mathrm{u}}
\Bigl(
A\cdot s(\langle\pi(\omega(I_H))\rangle_\sigma)\cdot A
\Bigr)
\subseteq
A\otimes I_H.
\]

Let
\[
m\in s(\langle\pi(\omega(I_H))\rangle_\sigma),
\qquad
a,a'\in A.
\]

By assumption,
\[
\mathrm{ver}_{\mathrm{u}}(m)\in A\otimes I_H.
\]

 Let us consider
\[
m= b'\mathrm{d}_{\mathrm{u}}b\in \Omega^1_{\mathrm{u}}(A).
\]

Then
\begin{align*}
\mathrm{ver}_{\mathrm{u}}(m)
&=
\mathrm{ver}_{\mathrm{u}}(b'\otimes b- b'b\otimes 1)
\\
&=
\mathrm{ver}_{\mathrm{u}}(b'\otimes b)- \mathrm{ver}_{\mathrm{u}}(b'b\otimes 1)
\\
&=
b'b_{(0)}\otimes b_{(1)}^{+}- b'b\otimes 1^{+}
\\
&=
b'b_{(0)}\otimes b_{(1)}^{+}.
\end{align*}

This implies that
\[
b'b_{(0)}\otimes b_{(1)}^{+}\in A\otimes I_H.
\]

Now,
\begin{align*}
\mathrm{ver}_{\mathrm{u}}(ma')
&=
\mathrm{ver}_{\mathrm{u}}(b'\otimes ba'- b'b\otimes a')
\\
&=
\mathrm{ver}_{\mathrm{u}}(b'\otimes ba')- \mathrm{ver}_{\mathrm{u}}(b'b\otimes a')
\\
&=
b'(ba')_{(0)}\otimes (ba')_{(1)}^{+}- b'ba'_{(0)}\otimes (a'_{(1)})^{+}
\\
&=
b'b_{(0)}a'_{(0)}\otimes (b_{(1)}a'_{(1)})^{+}- b'ba'_{(0)}\otimes (a'_{(1)})^{+}.
\end{align*}

Using the identity
\[
(h_1h_2)^+
=
h_1^+h_2^+
+
\varepsilon(h_1)h_2^+
+
\varepsilon(h_2)h_1^+, \quad\text{for } h_1, h_2\in H
\]
we put
\[
h_1= b_{(1)},
\qquad
h_2= a'_{(1)}.
\]
Hence we obtain
\begin{align*}
\mathrm{ver}_{\mathrm{u}}(ma')
&=
b'b_{(0)}a'_{(0)}
\otimes
\bigl(
b_{(1)}^{+}(a'_{(1)})^{+}
+
\varepsilon(b_{(1)})(a'_{(1)})^{+}
+
\varepsilon(a'_{(1)})b_{(1)}^{+}
\bigr)
\\
&\qquad
-
b'ba'_{(0)}\otimes (a'_{(1)})^{+}
\\
&=
b'b_{(0)}a'_{(0)}
\otimes
b_{(1)}^{+}(a'_{(1)})^{+}
+
b'b_{(0)}a'_{(0)}
\otimes
\varepsilon(b_{(1)})(a'_{(1)})^{+}
\\
&\qquad
+
b'b_{(0)}a'_{(0)}
\otimes
\varepsilon(a'_{(1)})b_{(1)}^{+}
-
b'ba'_{(0)}\otimes (a'_{(1)})^{+}
\\
&=
b'b_{(0)}a'_{(0)}
\otimes
b_{(1)}^{+}(a'_{(1)})^{+}
+
b'b_{(0)}\varepsilon(b_{(1)})a'_{(0)}
\otimes
(a'_{(1)})^{+}
\\
&\qquad
+
b'b_{(0)}a'_{(0)}\varepsilon(a'_{(1)})
\otimes
b_{(1)}^{+}
-
b'ba'_{(0)}\otimes (a'_{(1)})^{+}
\\
&=
b'b_{(0)}a'_{(0)}
\otimes
b_{(1)}^{+}(a'_{(1)})^{+}
+
b'ba'_{(0)}
\otimes
(a'_{(1)})^{+}
\\
&\qquad
+
b'b_{(0)}a'
\otimes
b_{(1)}^{+}
-
b'ba'_{(0)}\otimes (a'_{(1)})^{+}
\\
&=
b'b_{(0)}a'_{(0)}
\otimes
b_{(1)}^{+}(a'_{(1)})^{+}
+
b'b_{(0)}a'
\otimes
b_{(1)}^{+}.
\end{align*}

In the fourth equality we use the counit identities of the coaction:
\[
a'_{(0)}\varepsilon(a'_{(1)})= a',
\qquad
b_{(0)}\varepsilon(b_{(1)})= b.
\]

Now since
\[
\mathrm{ver}_{\mathrm{u}}(m)
=
b'b_{(0)}\otimes b_{(1)}^{+}\in A\otimes I_H,
\]
and $I_H$ is a right ideal of $H$, therefore,
\[
\mathrm{ver}_{\mathrm{u}}(ma')
=
b'b_{(0)}a'_{(0)}
\otimes
b_{(1)}^{+}(a'_{(1)})^{+}
+
b'b_{(0)}a'
\otimes
b_{(1)}^{+}
\in A\otimes I_H.
\]
 
 Now since $\mathrm{ver}_{\mathrm{u}}$ is left $A$--linear,
\[
\mathrm{ver}_{\mathrm{u}}(ama')
=
a\,\mathrm{ver}_{\mathrm{u}}(ma').
\]
Hence
\[
\mathrm{ver}_{\mathrm{u}}(ama')\in A\otimes I_H.
\]

Therefore,
\[
\mathrm{ver}_{\mathrm{u}}(N_A)\subseteq A\otimes I_H.
\]
Consequently, $\mathrm{ver}_{\mathrm{u}}$ descends to a well--defined map
\[
\mathrm{ver}:
\Omega^1(A)\to A\otimes H^+/I_H.
\]

Therefore, we have the following diagram
\begin{center}
\begin{tikzcd}
\Omega^1_{\mathrm{u}}(A) \arrow[rr, "\mathrm{ver}_{\mathrm{u}}"] \arrow[d, "q_{N_{A}}"'] &  & A\otimes H^{+} \arrow[d, "\mathrm{id}_{A}\otimes q_{I_{H}}"] \\
\Omega^1(A) \arrow[rr, "\mathrm{ver}"]                                           &  & A\otimes H^{+}/I_{H}                                        
\end{tikzcd}
\end{center}
Here, $q_{N_{A}}: \Omega^1_{\mathrm{u}}(A)\rightarrow \Omega^1(A)$ and $q_{I_H}: H^{+}\rightarrow H^{+}/I_{H}$  denote the cannonical quotient maps. We have 
\[
\mathrm{ver}(q_{N_{A}}(a\otimes a'))= (\mathrm{id}_{A}\otimes q_{I_H})\circ \mathrm{ver}_{\mathrm{u}}(a\otimes a')= aa'_{(0)}\otimes q_{I_{H}}((a'_{(1)})^{+})
\]
It is clear from the diagram that the surjectivity of the map $\mathrm{ver}$ follows from the surjectivity of the maps $\mathrm{ver}_{\mathrm{u}}$ and $q_{I_{H}}$. This completes the proof.
\end{proof}

The descent of the vertical map naturally raises the question whether
the associated connection $1$--form also descends to the quotient
calculus. The following proposition shows that this indeed holds for
$\sigma$--generated calculi.

\begin{proposition}
\label{thm:sigma-descent}
Let $A$ be a principal right $H$-comodule algebra, let
$I_H\subseteq H^+$ be a fixed right ideal, and let
\[
\Omega^1(A)=\Omega^1_{\mathrm{u}}(A)/N_A
\]
be a right $H$-covariant $\sigma$--generated first-order differential calculus on $A$ with respect to $I_H$. 
Then the connection $1$-form $\omega$ descends to a well-defined map
\[
\bar{\omega}:H^+/I_H\longrightarrow \Omega^1(A);
\]

\end{proposition}

\begin{proof}
Since $\Omega^1(A)$ is $\sigma$-generated, by definition one has
\[
N_A^{\mathrm{bal}}= A\cdot \langle \pi(\omega(I_H))\rangle_\sigma \cdot A .
\]
In particular, $\pi(\omega(h^+))\in N_A^{\mathrm{bal}}= \pi(N_A)$ for all $h^+\in I_H$. Since we know from the definition~\ref{def:sigma-generatedcalculi} that $\mathrm{ker}\, \pi\subset N_A$, this implies $[\omega(h^+)]_{\mathrm{ker}\,\pi}\in N_A/\mathrm{ker}\, \pi$, where $[\omega(h^+)]_{\mathrm{ker}\,\pi}$ denotes the corresponding quotient element of $\omega(h^+)$ in $N_A/\mathrm{ker}\, \pi$ under the cannonical quotient map $[\,\,\,]_{\mathrm{ker}\, \pi}: N_A\rightarrow N_A/\mathrm{ker}\, \pi$. It follows that $\omega(h^{+})\in N_A$ for all $h^+\in I_H$, and hence,
$\omega(I_H)\subset N_A$.
Therefore the connection $1$-form $\omega:H\to\Omega^1_{\mathrm{u}}(A)$
vanishes on $I_H$ modulo $N_A$ and hence descends uniquely to a linear map
\[
\bar{\omega}:H^+/I_H\longrightarrow \Omega^1(A),
\]
as claimed.

\end{proof}

\subsection{Functoriality of $\sigma$--generated calculi}

The construction of $\sigma$--generated calculi depends on the braided
vertical relations associated to a principal right $H$--comodule algebra,
together with the choice of a strong connection and a right
$H$--colinear splitting map. It is therefore natural to ask whether this
construction is compatible with morphisms of principal comodule
algebras preserving the underlying braided geometric data.

In this subsection we show that $\sigma$--generated calculi admit a
natural functorial formulation. More precisely, we introduce a category
whose objects consist of principal right $H$--comodule algebras equipped
with compatible connection and splitting data, and show that the
associated $\sigma$--generated calculi define a functor to the category
of right $H$--covariant first--order differential calculi.

\medskip

Let $H$ be a Hopf algebra and let $I_H\subseteq H^+$ be a fixed right ideal invariant under the right adjoint coaction of
$H$.

\medskip

We define a category $\mathcal{C}_{\sigma\text{--gen}}$ as follows.

\medskip

An object of $\mathcal{C}_{\sigma\text{--gen}}$ is a triple $(A,\omega,s_A)$, where:

\begin{enumerate}
\item $A$ is a principal right $H$--comodule algebra with $B=A^{\mathrm{co}(H)}\subset A$ a Hopf--Galois extension;

\item a strong connection $1$--form $\omega:H^+\to\Omega^1_{\mathrm{u}}(A)$;

\item a right $H$--colinear splitting map
\[
s_A:A\otimes_BA\to\Omega^1_{\mathrm{u}}(A)
\]
satisfying
\[
\pi_A\circ s_A=\mathrm{id}_{A\otimes_BA}.
\]
\end{enumerate}

For an object
\[
(A,\omega,s_A)\in\mathcal{C}_{\sigma\text{--gen}},
\]
the associated $\sigma$--generated differential calculus is defined by
\[
\Omega^1(A)
=
\Omega^1_{\mathrm{u}}(A)/N_A,
\]
where
\[
N_A
=
A\cdot
s_A(\langle\pi_A(\omega(I_H))\rangle_\sigma)
\cdot A
+
A\Omega^1_{\mathrm{u}}(B)A.
\]

\medskip

Let  $(A,\omega,s_A)$, $(A',\omega',s_{A'})$ be objects in $\mathcal{C}_{\sigma\text{--gen}}$.

\medskip

Let
\[
f:A\to A'
\]
be a right $H$--colinear algebra morphism.
Then $f$ canonically induces a morphism of universal differential
calculi
\[
f_*:\Omega^1_{\mathrm{u}}(A)\to\Omega^1_{\mathrm{u}}(A')
\]
defined by
\[
f_*(a\mathrm{d}_{\mathrm{u}}b)=f(a)\,\mathrm{d}_{\mathrm{u}}(f(b)).
\]

A morphism
\[
f\in
\mathrm{Hom}_{\mathcal{C}_{\sigma\text{--gen}}}
\bigl(
(A,\omega,s_A),
(A',\omega',s_{A'})
\bigr)
\]
is a right $H$--colinear algebra morphism
\[
f:A\to A'
\]
satisfying the following compatibility conditions:

\begin{enumerate}
\item
\[
(f\otimes_Bf)
\Bigl(
\langle\pi_A(\omega(I_H))\rangle_\sigma
\Bigr)
\subseteq
\langle\pi_{A'}(\omega'(I_H))\rangle_\sigma;
\]

\item
\[
f_*\circ s_A
=
s_{A'}\circ(f\otimes_Bf).
\]
\end{enumerate}

The second condition is expressed by the commutative diagram
\[
\begin{tikzcd}
A\otimes_BA
\arrow{r}{s_A}
\arrow{d}{f\otimes_Bf}
&
\Omega^1_{\mathrm{u}}(A)
\arrow{d}{f_*}
\\
A'\otimes_{B'}A'
\arrow{r}{s_{A'}}
&
\Omega^1_{\mathrm{u}}(A').
\end{tikzcd}
\]

\begin{proposition}\label{prop:morphism}
Let
\[
f\in
\mathrm{Hom}_{\mathcal{C}_{\sigma\text{--gen}}}
\bigl(
(A,\omega,s_A),
(A',\omega',s_{A'})
\bigr).
\]
Then $f$ induces a morphism of the corresponding
$\sigma$--generated differential calculi
\[
\Omega^1(A)\longrightarrow\Omega^1(A').
\]
\end{proposition}

\begin{proof}
Since
\[
f_*:\Omega^1_{\mathrm{u}}(A)\to\Omega^1_{\mathrm{u}}(A')
\]
is induced from the algebra morphism
\[
f:A\to A',
\]
it suffices to show that
\[
f_*(N_A)\subseteq N_{A'}.
\]

Let
\[
x\in
\langle\pi_A(\omega(I_H))\rangle_\sigma.
\]
By the first compatibility condition,
\[
(f\otimes_Bf)(x)
\in
\langle\pi_{A'}(\omega'(I_H))\rangle_\sigma.
\]

Using the second compatibility condition,
\[
f_*(s_A(x))
=
s_{A'}((f\otimes_Bf)(x)).
\]
Hence
\[
f_*(s_A(x))
\in
s_{A'}
\Bigl(
\langle\pi_{A'}(\omega'(I_H))\rangle_\sigma
\Bigr).
\]

Since $f_{*}$ is an $A$--bimodule morphism, therefore,
\[
f_*
\Bigl(
A\cdot
s_A(\langle\pi_A(\omega(I_H))\rangle_\sigma)
\cdot A
\Bigr)
\subseteq
A'\cdot
s_{A'}
(\langle\pi_{A'}(\omega'(I_H))\rangle_\sigma)
\cdot A'.
\]

Moreover,
\[
f_*(A\Omega^1_{\mathrm{u}}(B)A)
\subseteq
A'\Omega^1_{\mathrm{u}}(B')A'.
\]
Consequently,
\[
f_*(N_A)\subseteq N_{A'}.
\]

Hence $f_*$ descends to a well--defined morphism
\[
\Omega^1(A)
=
\Omega^1_{\mathrm{u}}(A)/N_A
\longrightarrow
\Omega^1_{\mathrm{u}}(A')/N_{A'}
=
\Omega^1(A').
\]
\end{proof}
Proposition~\ref{prop:morphism} shows that the construction of
$\sigma$--generated calculi is compatible with morphisms of principal
right $H$--comodule algebras preserving the braided vertical data.
Consequently, the assignment
\[
(A,\omega,s_A)
\longmapsto
\Omega^1(A)
\]
admits a natural categorical interpretation.

\begin{proposition}
 Let $\mathcal{D}$ be the category of right $H$--covariant first--order differential
calculi. Then the assignment
\[
\mathfrak{F}:
\mathcal{C}_{\sigma\text{--gen}}
\longrightarrow
\mathcal{D}
\]
defined on objects by
\[
\mathfrak{F}(A,\omega,s_A)
=
\Omega^1(A)
=
\Omega^1_{\mathrm{u}}(A)/N_A
\]
and on morphisms by
\[
\mathfrak{F}(f)
=
\overline{f_*},
\]
where
\[
\overline{f_*}:\Omega^1(A)\to\Omega^1(A')
\]
is the morphism induced by
\[
f_*:\Omega^1_{\mathrm{u}}(A)\to\Omega^1_{\mathrm{u}}(A'),
\]
defines a covariant functor.
\end{proposition}

\begin{proof}
By Proposition~\ref{prop:morphism},
every morphism
\[
f\in
\mathrm{Hom}_{\mathcal{C}_{\sigma\text{--gen}}}
\bigl(
(A,\omega,s_A),
(A',\omega',s_{A'})
\bigr)
\]
induces a well--defined morphism
\[
\overline{f_*}:
\Omega^1(A)\to\Omega^1(A').
\]
Hence $\mathfrak{F}$ is well defined on morphisms.

Let
\[
(A,\omega,s_A)\in\mathcal{C}_{\sigma\text{--gen}}.
\]
The identity morphism
\[
\mathrm{id}_A:A\to A
\]
induces the identity morphism
\[
(\mathrm{id}_A)_*
=
\mathrm{id}_{\Omega^1_{\mathrm{u}}(A)}.
\]
Therefore,
\[
\mathfrak{F}(\mathrm{id}_A)
=
\mathrm{id}_{\Omega^1(A)}.
\]

Next, let
\[
f:
(A,\omega,s_A)
\to
(A',\omega',s_{A'})
\]
and
\[
g:
(A',\omega',s_{A'})
\to
(A'',\omega'',s_{A''})
\]
be morphisms in
\[
\mathcal{C}_{\sigma\text{--gen}}.
\]
Since universal differential calculus is functorial,
\[
(g\circ f)_*
=
g_*\circ f_*.
\]
By using Proposition~\ref{prop:morphism}, we pass to the quotient calculi that yields
\[
\mathfrak{F}(g\circ f)
=
\overline{(g\circ f)_*}
=
\overline{g_*}\circ\overline{f_*}
=
\mathfrak{F}(g)\circ\mathfrak{F}(f).
\]
Hence $\mathfrak{F}$ preserves composition.

Therefore,
\[
\mathfrak{F}:
\mathcal{C}_{\sigma\text{--gen}}
\to
\mathcal{D}
\]
is a covariant functor.
\end{proof}

The functor
\[
\mathfrak F:
\mathcal C_{\sigma\text{--gen}}
\to
\mathcal D
\]
provides a categorical realization of the construction of
$\sigma$--generated calculi. In particular, it shows that the braided
quotient construction is natural with respect to morphisms of principal
right $H$--comodule algebras compatible with the chosen connection and
splitting data.

We now show that the descended vertical structure is compatible with
morphisms in the category $\mathcal{C}_{\sigma\text{--gen}}$. In particular, the vertical map associated to a $\sigma$--generated
calculus behaves naturally with respect to morphisms of principal right
$H$--comodule algebras.

\begin{proposition}
Let
\[
f:
(A,\omega,s_A)
\longrightarrow
(A',\omega',s_{A'})
\]
be a morphism in $\mathcal{C}_{\sigma\text{--gen}}$. Assume that the vertical maps
\[
\mathrm{ver}_{\mathrm{u}}:
\Omega^1_{\mathrm{u}}(A)\to A\otimes H^+,
\qquad
\mathrm{ver}'_{\mathrm{u}}:
\Omega^1_{\mathrm{u}}(A')\to A'\otimes H^+
\]
descend to well--defined maps
\[
\mathrm{ver}:
\Omega^1(A)\to A\otimes H^+/I_H,
\]
and
\[
\mathrm{ver}':
\Omega^1(A')\to A'\otimes H^+/I_H.
\]

Then the following diagram commutes:
\[
\begin{tikzcd}
\Omega^1(A)
\arrow{r}{\mathrm{ver}}
\arrow{d}{\overline{f_*}}
&
A\otimes H^+/I_H
\arrow{d}{f\otimes\mathrm{id}_{H^+/I_{H}}}
\\
\Omega^1(A')
\arrow{r}{\mathrm{ver}'}
&
A'\otimes H^+/I_H.
\end{tikzcd}
\]
Equivalently,
\[
(f\otimes\mathrm{id}_{H^+/I_{H}})
\circ
\mathrm{ver}
=
\mathrm{ver}'
\circ
\overline{f_*}.
\]
\end{proposition}

\begin{proof}
For
\[
a,b\in A,
\]
the vertical map satisfies
\[
\mathrm{ver}_{\mathrm{u}}(a\,\mathrm{d}_{\mathrm{u}}b)
=
ab_{(0)}\otimes b_{(1)}^+.
\]
Since
\[
f_*:
\Omega^1_{\mathrm{u}}(A)\to\Omega^1_{\mathrm{u}}(A')
\]
is induced by the right $H$--colinear algebra morphism
\[
f:A\to A',
\]
we obtain
\[
\begin{aligned}
\mathrm{ver}'_{\mathrm{u}}
\bigl(
f_*(a\,\mathrm{d}_{\mathrm{u}}b)
\bigr)
&=
\mathrm{ver}'_{\mathrm{u}}
\bigl(
f(a)\,\mathrm{d}_{\mathrm{u}}(f(b))
\bigr)
\\
&=
f(a)\,f(b)_{(0)}
\otimes
f(b)_{(1)}^+.
\end{aligned}
\]
Using right $H$--colinearity of $f$,
\[
\delta(f(b))
=
(f\otimes\mathrm{id}_{H})\delta(b),
\]
it follows that
\[
f(b)_{(0)}\otimes f(b)_{(1)}
=
f(b_{(0)})\otimes b_{(1)}.
\]
Hence
\[
\begin{aligned}
\mathrm{ver}'_{\mathrm{u}}
\bigl(
f_*(a\,\mathrm{d}_{\mathrm{u}}b)
\bigr)
&=
f(ab_{(0)})\otimes b_{(1)}^+
\\
&=
(f\otimes\mathrm{id}_{H^+})
\bigl(
ab_{(0)}\otimes b_{(1)}^+
\bigr)
\\
&=
(f\otimes\mathrm{id}_{H^+})
\bigl(
\mathrm{ver}_{\mathrm{u}}(a\,\mathrm{d}_{\mathrm{u}}b)
\bigr).
\end{aligned}
\]
Therefore,
\[
(f\otimes\mathrm{id}_{H^+})
\circ
\mathrm{ver}_{\mathrm{u}}
=
\mathrm{ver}'_{\mathrm{u}}
\circ
f_*.
\]

Since
\[
\mathrm{ver}_{\mathrm{u}}(N_A)\subseteq A\otimes I_H,
\qquad
\mathrm{ver}'_{\mathrm{u}}(N_{A'})\subseteq A'\otimes I_H,
\]
the vertical maps on the universal calculi induce well--defined quotient maps
\[
\mathrm{ver}:
\Omega^1(A)=\Omega^1_{\mathrm{u}}(A)/N_A
\longrightarrow
A\otimes H^+/I_H,
\]
and
\[
\mathrm{ver}':
\Omega^1(A')=\Omega^1_{\mathrm{u}}(A')/N_{A'}
\longrightarrow
A'\otimes H^+/I_H.
\]

Moreover, since by the Proposition~\ref{prop:morphism}
\[
f_*(N_A)\subseteq N_{A'},
\]
the morphism
\[
f_*:\Omega^1_{\mathrm{u}}(A)\to\Omega^1_{\mathrm{u}}(A')
\]
descends to a well--defined morphism
\[
\overline{f_*}:
\Omega^1(A)\to\Omega^1(A').
\]

Let 
\[
q_{N_{A}}: \Omega^1_{\mathrm{u}}(A)\rightarrow \Omega^1(A)
\]
and 
\[
q_{I_{H}}: H^+\rightarrow H^+/I_H
\]
be cannonical quotient maps. We consider $\eta\in \Omega^1_{\mathrm{u}}(A)$.

\medskip

Using
\[
(f\otimes\mathrm{id}_{H^+})
\circ
\mathrm{ver}_{\mathrm{u}}
=
\mathrm{ver}'_{\mathrm{u}}
\circ
f_*,
\]
we compute
\begin{align*}
&
((f\otimes \mathrm{id}_{H^+/I_H})\circ \mathrm{ver})(q_{N_{A}}(\eta))
\\
&=
(f\otimes \mathrm{id}_{H^+/I_H})
\Bigl(
(\mathrm{id}_{A}\otimes q_{I_{H}})
(\mathrm{ver}_{\mathrm{u}}(\eta))
\Bigr)
\\
&=
(\mathrm{id}_{A}\otimes q_{I_{H}})
\Bigl(
(f\otimes \mathrm{id}_{H^+})
(\mathrm{ver}_{\mathrm{u}}(\eta))
\Bigr)
\\
&=
(\mathrm{id}_{A'}\otimes q_{I_{H}})
\bigl(
\mathrm{ver'}_{\mathrm{u}}
\circ
f_{*}(\eta)
\bigr)
\\
&=
\mathrm{ver'}
\bigl(
q_{N_{A}}(f_{*}(\eta))
\bigr)
\\
&=
\mathrm{ver'}
\circ
\overline{f}_{*}(q_{N_{A}}(\eta)).
\end{align*}

Hence
\[
(f\otimes\mathrm{id}_{H^+/I_{H}})
\circ
\mathrm{ver}
=
\mathrm{ver}'
\circ
\overline{f_*}.
\]
\end{proof}

The preceding functoriality results show that $\sigma$--generated
calculi are naturally compatible with morphisms of principal right
$H$--comodule algebras. It is therefore natural to ask whether the
resulting quotient calculus admits a universal characterization among
right $H$--covariant calculi compatible with the braided vertical
relations determined by the ideal $I_H\subseteq H^+$. The following theorem shows that the associated $\sigma$--generated
calculus is universal among all admissible quotient calculi containing
the lifted braided vertical relations.

\begin{theorem}
\label{thm:universal-property}
Let $(A,\omega,s_A)$ be an object in $\mathcal{C}_{\sigma\text{--gen}}$.
We consider
\[
N_A
:=
A\cdot s_A(\langle
\pi_A(\omega(I_H))
\rangle_\sigma)\cdot A
+
A\Omega^1_{\mathrm{u}}(B)A.
\]
Then the associated $\sigma$--generated calculus
\[
\Omega^1(A)
=
\Omega^1_{\mathrm{u}}(A)/N_A
\]
is an initial object in the category of right $H$--covariant first--order differential calculi on $A$
satisfying $N_A\subseteq M$,
where $\Gamma^1(A)= \Omega^1_{\mathrm{u}}(A)/M$ is any right $H$--covariant first--order differential calculus on $A$.

More precisely, for every right $H$--covariant subbimodule $M\subset\Omega^1_{\mathrm{u}}(A)$
satisfying $N_A\subseteq M$,
there exists a unique surjective morphism
\[
\Phi:
\Omega^1(A)
\longrightarrow
\Gamma^1(A)
=
\Omega^1_{\mathrm{u}}(A)/M
\]
in the category $\mathcal{D}$ such that
\[
\Phi\circ q_{N_A}=q_M,
\]
where
\[
q_{N_A}:
\Omega^1_{\mathrm{u}}(A)\to\Omega^1(A),
\qquad
q_M:
\Omega^1_{\mathrm{u}}(A)\to\Gamma^1(A)
\]
denote the canonical quotient maps.
\end{theorem}

\begin{proof}
Let
\[
\eta\in N_A.
\]
Since
\[
N_A\subseteq M,
\]
it follows that
\[
q_M(\eta)=0.
\]
Hence
\[
N_A\subseteq\ker(q_M).
\]

Therefore, if
\[
q_{N_A}(\xi)=q_{N_A}(\xi') \quad\text{for } \xi,\xi'\in\Omega^1_{\mathrm{u}}(A)
\]
then it follows that
$\xi-\xi'\in N_A\subseteq\ker(q_M)$, since $q_{N_{A}}$ is a cannonical quotient map, so, $\mathrm{ker} (q_{N_{A}})= N_{A}$. As a consequence, we have
\[
q_M(\xi)=q_M(\xi').
\]

Thus the assignment
\[
\Phi(q_{N_A}(\xi))
:=
q_M(\xi),
\qquad
\xi\in\Omega^1_{\mathrm{u}}(A),
\]
defines a well--defined $A$--morphism
\[
\Phi:
\Omega^1(A)
=
\Omega^1_{\mathrm{u}}(A)/N_A
\longrightarrow
\Gamma^1(A)
=
\Omega^1_{\mathrm{u}}(A)/M.
\]
Since $N_A$ and $M$ are $A$--subbimodules of $\Omega^1_{\mathrm{u}}(A)$,
the quotient maps $q_{N_{A}}$ and $q_{M}$ are canonical $A$--bimodule morphisms, therefore, it follows that $\Phi$ is an $A$--bimodule morphism.
Moreover, for every $\xi\in\Omega^1_{\mathrm{u}}(A)$, we have
\[
(\Phi\circ q_{N_A})(\xi)
=
\Phi(q_{N_A}(\xi))
=
q_M(\xi).
\]
Hence
\[
\Phi\circ q_{N_A}=q_M.
\]

To show uniqueness, let
\[
\Psi:
\Omega^1(A)\to\Gamma^1(A)
\]
be another morphism satisfying
\[
\Psi\circ q_{N_A}=q_M.
\]
Then for every $\xi\in\Omega^1_{\mathrm{u}}(A)$, we obtain
\[
\Psi(q_{N_A}(\xi))
=
q_M(\xi)
=
\Phi(q_{N_A}(\xi)).
\]
Since $q_{N_A}$ is surjective, it follows that
\[
\Psi=\Phi.
\]

We shall show that $\Phi$ is surjective. Let $[\xi]_M\in\Gamma^1(A)$

\medskip

Since $q_M$ is surjective, there exists $\xi\in\Omega^1_{\mathrm{u}}(A)$
such that
\[
q_M(\xi)=[\xi]_M.
\]
Then
\[
[\xi]_M
=
q_M(\xi)
=
\Phi(q_{N_A}(\xi)),
\]
which proves that $\Phi$ is surjective.

\medskip

Finally, we shall show that $\Phi$ is a right $H$--colinear morphism.

\medskip

Since $N_A$ and $M$ are right $H$--covariant $A$--subbimodules of $\Omega^1_{\mathrm{u}}(A)$, the quotient calculi $\Omega^1(A)
=
\Omega^1_{\mathrm{u}}(A)/N_A$ and $\Gamma^1(A)
=
\Omega^1_{\mathrm{u}}(A)/M$ inherit induced right $H$--coactions
$\delta_{\Omega^1(A)}$, $\delta_{\Gamma^1(A)}$ respectively.

\medskip

Let $\xi\in\Omega^1_{\mathrm{u}}(A)$.

Then
\[
\begin{aligned}
\delta_{\Gamma^1(A)}
\bigl(
\Phi(q_{N_A}(\xi))
\bigr)
&=
\delta_{\Gamma^1(A)}
\bigl(
q_M(\xi)
\bigr)
\\
&=
(q_M\otimes\mathrm{id}_{H})
\bigl(
\delta_{\Omega^1_{\mathrm{u}}(A)}(\xi)
\bigr),
\end{aligned}
\]
since $q_M$ is right $H$--colinear.

\medskip

On the other hand,
\[
\begin{aligned}
((\Phi\otimes\mathrm{id}_{H})
\circ
\delta_{\Omega^1(A)})
(q_{N_A}(\xi))
&=
(\Phi\otimes\mathrm{id}_{H})
\bigl(
(q_{N_A}\otimes\mathrm{id}_{H})
(\delta_{\Omega^1_{\mathrm{u}}(A)}(\xi))
\bigr)
\\
&=
(q_M\otimes\mathrm{id}_{H})
(\delta_{\Omega^1_{\mathrm{u}}(A)}(\xi)),
\end{aligned}
\]
In the first equality, we use the fact that $q_{N_{A}}$ is right $H$--colinear, and in the second equality, we use $\Phi\circ q_{N_A}=q_M$.

\medskip

Therefore,
\[
\delta_{\Gamma^1(A)}
\circ
\Phi
=
(\Phi\otimes\mathrm{id}_{H})
\circ
\delta_{\Omega^1(A)}.
\]
Hence, $\Phi$ is right $H$--colinear, and therefore a morphism in the category $\mathcal D$.

\end{proof} 

Theorem~\ref{thm:universal-property} shows that the associated $\sigma$--generated calculus is
characterized by a universal factorization property among right
$H$--covariant quotient calculi compatible with the braided vertical
relations determined by the ideal $I_H\subseteq H^+$.
In this sense, $\sigma$--generated calculi provide a canonical
first--order braided differential structure naturally associated to the
underlying principal comodule algebra and its vertical geometry.

\section{Examples of $\sigma$--generated differential calculi}

In this section we illustrate the general construction of
$\sigma$--generated differential calculi in the setting of quantum
principal bundles arising from quantum projective spaces and quantum
lens spaces. These examples show that the abstract braided construction
developed in the previous section naturally appears in important classes
of noncommutative Hopf--Galois extensions.

In particular, we show that the braided vertical relations determined
by the Durdević braiding and strong connection data give rise to
explicit right $H$--covariant differential calculi on quantum unitary
groups and quantum lens spaces. The resulting calculi are naturally
adapted to the underlying quantum fibration structures.

\medskip

We begin with the quantum Hopf fibration
\[
\mathcal{O}_q(\mathbb{CP}^n)
\subseteq
\mathcal{O}_q(SU_{n+1}),
\]
which provides one of the fundamental examples of a quantum principal
bundle. We show that the general construction of
$\sigma$--generated calculi naturally produces a braided differential
calculus compatible with the geometry of quantum projective space.

\subsection{$\sigma$--generated calculi on $\mathcal{O}_q(SU_{n+1})$ over $\mathcal{O}_q(\mathbb{CP}^n)$}

Let
\[
A:=\mathcal{O}_q(SU_{n+1}),
\qquad
H:=\mathcal{O}(U(1))
=
\mathbb{C}[z,z^{-1}],
\]
and let
\[
B:=A^{\mathrm{co}(H)}
=
\mathcal{O}_q(\mathbb{CP}^n).
\]
The right $H$--coaction on $A$ is induced from the canonical
$\mathbb{Z}$--grading 
\[
A=\bigoplus_{k\in\mathbb{Z}}A_k,
\qquad
\delta(a_k)=a_k\otimes z^k,
\]
for all $a_k\in A_k$.
It is well known that $B\subset A$ is a principal right $H$--comodule algebra \cite{brzezinski1998coalgebra, hajac1996strong}.

Since
\[
H^+=\ker\varepsilon
=
\langle z-1\rangle,
\]
we fix the right ideal
\[
I_H:=\langle z-1\rangle\subseteq H^+.
\]
Moreover,
\[
\mathrm{Ad}_R(z-1)
=
(z-1)\otimes z+1\otimes(z-1)
\in I_H\otimes H,
\]
hence $I_H$ is invariant under the right adjoint coaction of $H$.

Let
\[
\ell:H\to A\otimes A
\]
be a strong connection on the quantum Hopf fibration and let
\[
\omega:H^+\to\Omega^1_{\mathrm{u}}(A)
\]
be the associated connection $1$--form
\[
\omega(h^+)
=
\ell(h)-\varepsilon(h)\,1\otimes1.
\]

Denote by
\[
\tau:H\to A\otimes_BA
\]
the translation map
\[
\tau(h)
=
\mathrm{can}^{-1}(1\otimes h).
\]
Then 
(for example, see \cite{hajac1996strong})
\[
\pi(\omega(h^+))
=
\tau(h)-\varepsilon(h)\,1\otimes_B1.
\]
In particular,
\[
\pi(\omega(z-1))
=
\tau(z)-1\otimes_B1.
\]

Define
\[
W_{I_H}
:=
\left\langle
\pi(\omega(I_H))
\right\rangle_\sigma
=
\left\langle
\tau(z)-1\otimes_B1
\right\rangle_\sigma.
\]

Since $H=\mathcal{O}(U(1))$ is cosemisimple, the category of right $H$--comodules is semisimple.
Hence the short exact sequence of right $H$--comodules \cite{brzezinski1998coalgebra, hajac1996strong}.
\[
0
\longrightarrow
\ker\pi
\longrightarrow
\Omega^1_{\mathrm{u}}(A)
\xrightarrow{\;\pi\;}
A\otimes_BA
\longrightarrow
0
\]
splits in the category of right $H$--comodules.
Therefore, there exists a right $H$--colinear splitting
\[
s:A\otimes_BA\to\Omega^1_{\mathrm{u}}(A)
\]
such that
\[
\pi\circ s=\mathrm{id}_{A\otimes_BA}.
\]

We define
\[
N_A
:=
A\cdot s(W_{I_H})\cdot A
+
A\Omega^1_{\mathrm{u}}(B)A.
\]
Then $N_A$ is a right $H$--covariant $A$--subbimodule of
$\Omega^1_{\mathrm{u}}(A)$.
Hence by the Proposition~\ref{prop:existence-sigma-generated}, the quotient
\[
\Omega^1(A)
=
\Omega^1_{\mathrm{u}}(A)/N_A
\]
defines a right $H$--covariant $\sigma$--generated first--order
differential calculus on $A=\mathcal{O}_q(SU_{n+1})$.

\begin{remark}
The element
\[
\tau(z)-1\otimes_B1
\]
encodes the vertical quantum $U(1)$--direction of the Hopf fibration.
Thus the above construction may be viewed as a braided extension of the
horizontal calculus on $\mathcal{O}_q(\mathbb{CP}^n)$ by a vertical
quantum circle relation.

Since the canonical covariant differential calculus on $\mathcal{O}_q(\mathbb{CP}^n)$ is $2n$--dimensional, the associated $\sigma$--generated calculus on $\mathcal{O}_q(SU_{n+1})$ may be naturally interpreted as a braided
$(2n+1)$--dimensional quantum fibration calculus consisting of a
horizontal quantum projective space component together with a single
vertical quantum circle direction.
\end{remark}
We next consider quantum lens spaces, which provide nontrivial finite
quotients of quantum odd spheres and form another important class of
quantum principal bundles. The following example shows that the
construction of $\sigma$--generated calculi is stable under these
noncommutative quotient constructions and extends naturally to the
associated quantum lens space fibrations.

\subsection{$\sigma$--generated calculi on quantum lens spaces}

Let
\[
A:=\mathcal{O}_q(\mathcal{L}^{(n,r)}),
\qquad
B:=\mathcal{O}_q(\mathbb{CP}^n),
\]
where
\[
\mathcal{O}_q(\mathcal{L}^{(n,r)})
=
\bigoplus_{N\in\mathbb{Z}}\mathcal{L}_{rN}
\]
denotes the coordinate algebra of the quantum lens space of dimension
$2n+1$ and index $r$ and $\mathcal{L}_{rN}$ denotes line bundle over $\mathcal{O}_q(\mathbb{CP}^n)$ (for example, see \cite{arici2016gysin, arici2016pimsner}).
The algebra inclusion 
\[
j:B\hookrightarrow A
\]
is a quantum principal bundle with structure group $H:= \mathcal{O}(\widetilde{U(1)})$, where,
\[
\widetilde{U(1)}
:=
U(1)/\mathbb{Z}_r,
\]
and
\[
B
=
A^{\mathrm{co}(H)}.
\]

\medskip

Since $\widetilde{U(1)}
\cong
U(1)$ as compact abelian Lie groups, it follows that the Hopf algebra $H=\mathcal{O}(\widetilde{U(1)})$ is cosemisimple. Hence the category of right $H$--comodules is semisimple. Therefore, there exists a right $H$--colinear splitting
\[
s:A\otimes_BA\to\Omega^1_{\mathrm{u}}(A)
\]
such that
\[
\pi\circ s=\mathrm{id}_{A\otimes_BA}.
\]

Let
\[
I_H:=\langle u-1\rangle\subseteq H^+,
\]
where
\[
H\simeq\mathbb{C}[u,u^{-1}].
\]
Since $I_H$ is invariant under the right adjoint coaction, the general existence
Proposition~\ref{prop:existence-sigma-generated} of $\sigma$--generated calculi applies. Hence the associated braided vertical
subspace
\[
W_{I_H}
=
\left\langle
\tau(u)-1\otimes_B1
\right\rangle_\sigma
\]
determines a right $H$--covariant $\sigma$--generated first--order
differential calculus on
\[
\mathcal{O}_q(\mathcal{L}^{(n,r)}).
\]

Thus the construction of $\sigma$--generated calculi extends naturally
from quantum projective fibrations to quantum lens space fibrations.
Geometrically, the resulting calculus may be viewed as a braided
extension of the canonical differential calculus on $\mathcal{O}_q(\mathbb{CP}^n)$ by a single vertical quantum circle direction associated to the lens
space fibration.

\medskip

The above examples show that $\sigma$--generated differential calculi
occur naturally in important classes of quantum principal bundles
arising in noncommutative geometry. In particular, the construction is
compatible with both quantum projective space fibrations and quantum
lens space fibrations, illustrating the flexibility of the braided
framework developed in this paper.


\section*{Acknowledgements}
The author would like to thank Thomas Weber, Emanuele Latini, Antonio Del Donno and Giovanni Gava for helpful discussions.


\bibliographystyle{plain}
\bibliography{name1}

\end{document}